\documentclass[smallextended,numbook]{svjour3_mod}

\usepackage{amsmath}
\usepackage{amssymb}

\usepackage{mathptmx}
\usepackage{helvet}
\usepackage{courier}

\usepackage{graphicx}

\journalname{Numerische Mathematik}

\begin{document}

\title{On the expansion of solutions of Laplace-like equations\\ 
 into traces of separable higher dimensional functions}

\author{Harry Yserentant}

\institute{
 Technische Universit\"at Berlin, Institut f\"ur Mathematik
 10623 Berlin, Germany\\
\email{yserentant@math.tu-berlin.de}}

\date{August 20, 2020}

\titlerunning{
 On the expansion of solutions of Laplace-like equations into \ldots}
 
\authorrunning{H. Yserentant}

\maketitle


\begin{abstract}
This paper deals with the equation $-\Delta u+\mu u=f$
on high-dimensional spaces $\mathbb{R}^m$ where $\mu$ 
is a positive constant. If the right-hand side $f$ is 
a rapidly converging series of separable functions, 
the solution $u$ can be represented in the same way. 
These constructions are based on approximations of the 
function $1/r$ by sums of exponential functions. The 
aim of this paper is to prove results of similar kind 
for more general right-hand sides $f(x)=F(Tx)$ that 
are composed of a separable function on a space of a 
dimension $n$ greater than $m$ and a linear mapping 
given by a matrix $T$ of full rank. These results are 
based on the observation that in the high-dimensional 
case, for $\omega$ in most of the $\mathbb{R}^n$, the 
euclidian norm of the vector $T^t\omega$ in the lower 
dimensional space $\mathbb{R}^m$ behaves like the 
euclidian norm of $\omega$.

%
%
%
%
%

              
%
%
%
%
%
%

\end{abstract}


\renewcommand {\thefigure}{\arabic{figure}}

\newcommand   {\rmref}[1]   {{\rm (\ref{#1})}}

\newcommand   {\fourier}[1] {\widehat{#1}}
\newcommand   {\diff}[1]    {\mathrm{d}#1}

\def \xy      {\Big(\begin{matrix}x\\y\end{matrix}\Big)}
\def \oe      {\Big(\begin{matrix}\omega\\\eta\end{matrix}\Big)}

\def \wT      {\widetilde{T}}

\def \dx      {\,\diff{x}}
\def \domega  {\,\diff{\omega}}
\def \deta    {\,\diff{\eta}}

\def \dr      {\,\diff{r}}
\def \ds      {\,\diff{s}}
\def \dt      {\,\diff{t}}

\def \FL      {W_0}

\def \L       {\mathcal{L}}


\section{Introduction}

The approximation of high-dimensional functions, whether they 
be given explicitly or implicitly as solutions of differential 
equations, represents a grand challenge for applied mathematics. 
High-dimensional problems arise in many fields of application 
such as data analysis and statistics, but first of all in the 
natural sciences. The Schr\"odinger equation, which links 
chemistry to physics and describes a system of electrons and 
nuclei that interact by Coulomb attraction and repulsion 
forces, forms an important example. The present work is 
partly motivated by applications in the context of quantum 
theory and is devoted to the equation
\begin{equation}    \label{eq1.1}
-\Delta u+\mu u=f
\end{equation}
on $\mathbb{R}^m$ for high dimensions $m$, with $\mu>0$ 
a given constant. Provided the right-hand side $f$ of 
this equation possesses an integrable Fourier transform,  
\begin{equation}    \label{eq1.2}
u(x)=\Big(\frac{1}{\sqrt{2\pi}}\Big)^m\!\int
\frac{1}{\mu+\|\omega\|^2}\,\fourier{f}(\omega)\,
\mathrm{e}^{\,\mathrm{i}\,\omega\cdot x}\domega
\end{equation}
is a solution of this equation, and the only solution
that tends uniformly to zero as x goes to infinity.
If the right-hand side $f$ of the equation is a tensor 
product
\begin{equation}    \label{eq1.3}
f(x)=\prod_i\phi_i(x_i)
\end{equation}
of univariate functions or a rapidly converging series 
of such tensor products, the same holds for the Fourier 
transform of $f$. If one replaces the corresponding 
term in the high-dimensional integral (\ref{eq1.2})
by an approximation 
\begin{equation}    \label{eq1.4}
\frac{1}{r}\approx\sum_k a_k\mathrm{e}^{-\beta_kr},
\quad\frac{1}{\mu+\|\omega\|^2}\approx
\sum_k a_k\mathrm{e}^{-\beta_k\mu}
\prod_i\mathrm{e}^{-\beta_k\omega_i^2},
\end{equation}
the integral collapses in this case therefore to a 
sum of products of one-dimensional integrals. That is, 
the solution of the equation can, independent of the 
space dimension, be approximated by a finite or 
infinite sum of such tensor products. Usually one 
starts from approximations of $1/r$ of given absolute 
accuracy. Such approximations of $1/r$ are studied in 
\cite{Braess-Hackbusch} and \cite{Braess-Hackbusch_2} 
and result in error estimates in terms of the 
right-hand side of the equation. By reasons that will 
become clear later, we will focus in the present paper 
on approximations of $1/r$ of given relative 
accuracy. In the context here, they lead to error 
estimates in terms of the solution of the equation 
itself. An example of such an approximation in form 
of an infinite series is
\begin{equation}    \label{eq1.5}
\frac{1}{r}\approx
h\sum_{k=-\infty}^\infty\mathrm{e}^{kh}
\exp(-\mathrm{e}^{kh}r).
\end{equation}
It arises from an integral representation of $1/r$
that is discretized by the trapezoidal or midpoint 
rule. It has been analyzed in 
\cite[Sect. 5]{Scholz-Yserentant} and is extremely 
accurate. The relative error tends exponentially 
with the distance $h$ of the nodes to zero. It is 
less than $5\cdot 10^{-8}$ for $h=1/2$, and for 
$h=1$ still less than $7\cdot 10^{-4}$. The 
approximation properties of partial sums of this 
series will be studied later.

The conclusion is that a tensor product structure like 
(\ref{eq1.3}) of the right-hand side of the equation 
is directly reflected in its solution. This effect 
enables the solution of truly high-dimensional 
equations \cite{Grasedyck}, \cite{Khoromskij} and 
probably forms one of the bases 
\cite{Dahmen-DeVore-Grasedyck-Sueli} for the success 
of modern tensor product methods \cite{Hackbusch}. 
The aim of the present paper is to generalize this 
observation to right-hand sides 
\begin{equation}    \label{eq1.6}
f(x)=F(Tx)
\end{equation}
that are composed of a separable function $F$ on a 
space of a dimension $n$ greater than the original 
dimension $m$ and a linear mapping given by a matrix 
$T$ of full rank. This covers, for example, 
right-hand sides $f$ that depend explicitly on 
differences of the components of $x$. We will prove
that the solution $u$ of the equation (\ref{eq1.1}) 
can in such cases almost always be well approximated 
by finite sums of functions of the same type, provided 
the ratio $n/m$ of the dimensions does not become too 
large. Background is some kind of concentration of 
measure effect in high space dimensions. Our main 
tool is the representation $u(x)=U(Tx)$ of the 
solution in terms of the solution $U$ of a 
degenerate elliptic equation $\L U=F$ in the higher 
dimensional space. Approximations to $U$ are then 
iteratively generated.


\section{Functions with integrable Fourier transform
 and their traces}
 
We consider in this paper functions 
$u:\mathbb{R}^d\to\mathbb{R}$, $d$ a varying and 
potentially high dimension, that possess a then 
also unique representation
\begin{equation}    \label{eq2.1}
u(x)=\Big(\frac{1}{\sqrt{2\pi}}\Big)^d
\!\int\fourier{u}(\omega)\,
\mathrm{e}^{\,\mathrm{i}\,\omega\cdot x}\domega
\end{equation}
in terms of a function $\fourier{u}\in L_1(\mathbb{R}^d)$,
their Fourier transform. Such functions are uniformly 
continuous and tend uniformly to zero as $x$ goes to 
infinity, the Riemann-Lebesgue theorem. The space 
$\FL(\mathbb{R}^d)$ of these functions becomes under 
the norm
\begin{equation}    \label{eq2.2}   
\|u\|=\Big(\frac{1}{\sqrt{2\pi}}\Big)^d\!\int
|\fourier{u}(\omega)|\domega
\end{equation}
a Banach space and even a Banach algebra. The norm 
(\ref{eq2.2}) dominates the maximum norm of the 
functions in this space. If the functions
$(\mathrm{i}\omega)^\beta\,\fourier{u}(\omega)$, 
$\beta\leq\alpha$, in multi-index notation, are 
integrable as well, the partial derivative 
$\mathrm{D}^\alpha u$ of $u$ exists, is given by
\begin{equation}    \label{eq2.3}    
(\mathrm{D}^\alpha u)(x)=\Big(\frac{1}{\sqrt{2\pi}}\Big)^d\!
\int(\mathrm{i}\omega)^\alpha\,\fourier{u}(\omega)\,
\mathrm{e}^{\,\mathrm{i}\,\omega\cdot x}\domega,
\end{equation}
and is as the function $u$ itself uniformly 
continuous and vanishes as $u$ at infinity. 
For partial derivatives of first order, this 
follows from the Fourier representation of 
the corresponding difference quotients and 
the dominated convergence theorem, and, for 
derivatives of higher order, it follows 
by induction.

Let $T$ be a from now on fixed $(n\times m)$-matrix 
of full rank $m<n$ and let
\begin{equation}    \label{eq2.4}
u:\mathbb{R}^m\to\mathbb{R}:x\to U(Tx)
\end{equation}
be the trace of a function $U:\mathbb{R}^n\to\mathbb{R}$
with an integrable Fourier transform. We first calculate 
the Fourier transform of such trace functions. 
\begin{theorem}     \label{thm2.1}
Let $U:\mathbb{R}^m\times\mathbb{R}^{n-m}\to\mathbb{R}$ be 
a function with an integrable Fourier transform. Its trace 
function \rmref{eq2.4} possesses then an integrable Fourier
transform as well. It reads, in terms of the Fourier
transform of $U$, 
\begin{equation}    \label{eq2.5}
\fourier{u}(\omega)=\Big(\frac{1}{\sqrt{2\pi}}\Big)^{n-m}\!
\int\frac{1}{|\det\wT|}\,\fourier{U}\Big(\wT^{-t}\oe\Big)\deta,
\end{equation}
where $\wT$ is an arbitrary invertible matrix of dimension 
$n\times n$ whose first $m$ columns coincide with those of 
$T$. The norm \rmref{eq2.2} of the trace function satisfies 
the estimate
\begin{equation}    \label{eq2.6}
\|u\|\leq\|U\|
\end{equation}
in terms of the corresponding norm of the function $U$.
\end{theorem}
\begin{proof}
The scaled $L_1$-norm
\begin{displaymath}
\Big(\frac{1}{\sqrt{2\pi}}\Big)^m\!\int
|\fourier{u}(\omega)|\domega
=\Big(\frac{1}{\sqrt{2\pi}}\Big)^n\!\int\bigg|\int
\frac{1}{|\det\wT|}\,\fourier{U}\Big(\wT^{-t}\oe\Big)\deta
\bigg|\domega
\end{displaymath}
of the function (\ref{eq2.5}) remains, by 
Fubini's theorem and the transformation 
theorem for multivariate integrals, finite 
and satisfies the estimate
\begin{displaymath}
\Big(\frac{1}{\sqrt{2\pi}}\Big)^m\!\int
|\fourier{u}(\omega)|\domega \leq
\Big(\frac{1}{\sqrt{2\pi}}\Big)^n\!\int
\Big|\fourier{U}\oe\Big|\,\diff(\omega,\eta).
\end{displaymath}
That the function (\ref{eq2.5}) is the Fourier 
transform of the trace function follows from 
\begin{displaymath}
\omega\cdot x=\wT^{-t}\oe\cdot\wT
\Big(\begin{matrix}x\\0\end{matrix}\Big)
=\wT^{-t}\oe\cdot Tx
\end{displaymath}
and again Fubini's theorem and the
transformation theorem. 
\qed
\end{proof}
The estimate (\ref{eq2.6}) is sharp. Every function
$u$ in $\FL(\mathbb{R}^m)$ is trace of a function
$U$ in $\FL(\mathbb{R}^n)$ with norm $\|U\|=\|u\|$, 
for example of that with the Fourier transform
\begin{equation}    \label{eq2.7}
\fourier{U}\oe=|\det\wT|\,V\Big(\wT^t\oe\Big), \quad
V\oe=\fourier{u}(\omega)\mathrm{e}^{-\frac12\|\eta\|^2}.
\end{equation}
A consequence of Theorem~\ref{thm2.1} is that 
the traces of functions with Fourier transform 
vanishing outside of a strip around the kernel 
of $T^t$ are bandlimited.
\begin{lemma}       \label{lm2.1}
Let the Fourier transform $\fourier{U}\in L_1$ of 
the function $U$ vanish outside of the set of all 
$\omega$ for which $\|T^t\omega\|\leq\Omega$ holds
in a given norm. Then the Fourier transform of its 
trace function vanishes outside of the ball of 
radius $\Omega$ around the origin.
\end{lemma}
\begin{proof}
We split the vectors in $\mathbb{R}^n$ again 
into the parts $\omega\in\mathbb{R}^m$ and 
$\eta\in\mathbb{R}^{n-m}$, as in 
Theorem~\ref{thm2.1} and its proof. Because 
of our assumption on the support of 
$\fourier{U}$, 
\begin{displaymath}
\|\omega\|=\Big\|T^t\wT^{-t}\oe\Big\|\leq\Omega
\end{displaymath}
holds for the $\omega$ and $\eta$ for which 
the integrand in the representation (\ref{eq2.5}) 
of the Fourier transform $\fourier{u}(\omega)$ 
of the trace function takes a value different 
from zero, which means that $\fourier{u}(\omega)$ 
must vanish for arguments $\omega$ of norm 
$\|\omega\|>\Omega$.  
\qed
\end{proof}


\section{Shifted Laplace equations with trace 
 functions as right-hand sides}

We now return to the equation $-\Delta u+\mu u=f$  
from Sect.~1. We show that its solution can, for 
a right-hand side $f(x)=F(Tx)$ that is trace of a 
function $F:\mathbb{R}^n\to\mathbb{R}$ with Fourier 
transform in $L_1$, be written as trace $u(x)=U(Tx)$ 
of a function $U:\mathbb{R}^n\to\mathbb{R}$ that 
solves a degenerate elliptic equation.
\begin{theorem}     \label{thm3.1}
Let $f:\mathbb{R}^m\to\mathbb{R}$ be a function 
with Fourier transform in $L_1$ and let $\mu$ 
be a positive constant. The twice continuously 
differentiable function
\begin{equation}    \label{eq3.1}
u(x)=\Big(\frac{1}{\sqrt{2\pi}}\Big)^m\!\int
\frac{1}{\mu+\|\omega\|^2}\,\fourier{f}(\omega)\,
\mathrm{e}^{\,\mathrm{i}\,\omega\cdot x}\domega,
\end{equation}
with $\|\omega\|$ the euclidian norm of $\omega$,
is then the only solution of the equation
\begin{equation}    \label{eq3.2}
-\Delta u+\mu u=f
\end{equation}
on the $\mathbb{R}^m$ that vanishes at infinity.
\end{theorem}
\begin{proof}
That $u$ is a twice continuously differentiable
function that solves the equation and vanishes 
at infinity follows from the remarks made in 
the last section on functions with integrable 
Fourier transform. The maximum principle states 
that it is the only solution of the equation 
with this property.
\qed
\end{proof}
The solution (\ref{eq3.1}) of the equation 
(\ref{eq3.2}) can also be characterized in 
a different way. We call a function 
$u\in\FL(\mathbb{R}^m)$ a weak solution of 
this equation if
\begin{equation}    \label{eq3.3}
\int u\,(-\Delta\varphi+\mu\varphi)\dx
=\int f\varphi\dx
\end{equation}
holds for all rapidly decreasing functions 
$\varphi$.

\begin{lemma}       \label{lm3.1}
The function \rmref{eq3.1} is the only weak
solution of the equation \rmref{eq3.2}.
\end{lemma}
\begin{proof}
That the function (\ref{eq3.1}) is a weak solution
of the equation (\ref{eq3.2}) follows from
\begin{displaymath}
\int \mathrm{e}^{\,\mathrm{i}\,\omega\cdot x}
(-\Delta\varphi+\mu\varphi)(x)\dx
=\int \big(\mu+\|\omega\|^2\big)
\mathrm{e}^{\,\mathrm{i}\,\omega\cdot x}
\varphi(x)\dx
\end{displaymath}
and Fubini's theorem, which is applied here twice, 
first to exchange the order of integration with 
respect to $x$ and $\omega$, and then to revert 
this process. If $u_1$ and $u_2$ are weak 
solutions of the equation, we have
\begin{displaymath}
\int(u_1-u_2)(-\Delta\varphi+\mu\varphi)\dx=0
\end{displaymath}
for all rapidly decreasing functions $\varphi$.
As the equation $-\Delta\varphi+\mu\varphi=\chi$
possesses for all rapidly decreasing functions
$\chi$ a rapidly decreasing solution $\varphi$,
with a Fourier representation as above, for all 
rapidly decreasing functions $\chi$ then
\begin{displaymath}
\int(u_1-u_2)\chi\dx=0
\end{displaymath}
holds. The difference of $u_1$ and $u_2$ 
must therefore vanish.
\qed
\end{proof}

Let the right-hand side now be the trace $f(x)=F(Tx)$ 
of a function $F$ in $\FL(\mathbb{R}^n)$. As such, it 
is by the results of the previous section a function 
in  $\FL(\mathbb{R}^m)$. The crucial observation is 
that we can lift the equation from $\mathbb{R}^m$ 
into $\mathbb{R}^n$.
\begin{theorem}     \label{thm3.2}
The solution \rmref{eq3.1} is the trace 
$u(x)=U(Tx)$ of the function
\begin{equation}    \label{eq3.4}
U(y)=\Big(\frac{1}{\sqrt{2\pi}}\Big)^n\!\int
\frac{1}{\mu+\|T^t\omega\|^2}\,\fourier{F}(\omega)\,
\mathrm{e}^{\,\mathrm{i}\,\omega\cdot y}\domega
\end{equation}
mapping the higher dimensional space to 
the real numbers.
\end{theorem}
\begin{proof}
Using again Fubini's theorem in the previously described
manner and observing that for rapidly decreasing functions
$\varphi$ because of $\omega\cdot Tx=T^t\omega\cdot x$
\begin{displaymath}
\int \mathrm{e}^{\,\mathrm{i}\,\omega\cdot Tx}
(-\Delta\varphi+\mu\varphi)(x)\dx
=\int \big(\mu+\|T^t\omega\|^2\big)
\mathrm{e}^{\,\mathrm{i}\,\omega\cdot Tx}
\varphi(x)\dx
\end{displaymath}
holds, one recognizes that the trace function 
$u$ is a weak solution of equation (\ref{eq3.2}). 
As such, it coincides by Lemma~\ref{lm3.1}
with the solution (\ref{eq3.1}) of this 
equation.
\qed
\end{proof}
The function (\ref{eq3.4}) is in the domain 
of the operator $\L$ given by
\begin{equation}    \label{eq3.5}
(\L U)(y)=\Big(\frac{1}{\sqrt{2\pi}}\Big)^n\!\int
\big(\mu+\|T^t\omega\|^2\big)\,\fourier{U}(\omega)\,
\mathrm{e}^{\,\mathrm{i}\,\omega\cdot y}\domega.
\end{equation}
By definition, it solves the equation
\begin{equation}    \label{eq3.6}
\L U=F.
\end{equation}
Because the expression $\mu+\|T^t\omega\|^2$ is 
a second order polynomial in the components of 
$\omega$, $\L$ can be considered as a second 
order differential operator and this equation  
therefore as a degenerate elliptic equation.
If the Fourier transform of $F$ and then also 
that of the solution $U$ have a bounded support,  
$U$ is infinitely differentiable and its 
derivatives can be obtained by differentiation 
under the integral sign. In this case, the 
function (\ref{eq3.4}) is a classical solution 
of the equation (\ref{eq3.6}), which is, however,
irrelevant for the following considerations. 

Instead of attacking the original equation 
(\ref{eq3.2}) directly, we will approximate 
the solution of the higher dimensional equation 
(\ref{eq3.6}). We are here primarily interested 
in right-hand sides $F$ that are products of 
lower-dimensional functions or sums or rapidly 
converging series of such functions. Because 
$TT^t$ will only in exceptional cases be a 
diagonal matrix, the function
\begin{equation}    \label{eq3.7}
\frac{1}{\mu+\|T^t\omega\|^2}
\end{equation}
can, however, in general not be approximated 
as easily by sums of separable Gauss functions 
as sketched in the introduction for its 
counterpart in the representation (\ref{eq3.1}) 
of the solution of the original equation. 
We will show that this problem vanishes in the 
high-dimensional case due to a concentration 
of measure effect.


\section{The iterative solution of the degenerate 
 elliptic equation}

To start with, let $\alpha:\mathbb{R}^n\to\mathbb{R}$ 
be a measurable, bounded function for which
\begin{equation}    \label{eq4.1}
\Big|\,1-\alpha(\omega)\big(\mu+\|T^t\omega\|^2\big)\,\Big|
\,<\,1
\end{equation}
holds for all $\omega\in\mathbb{R}^n$ and assign 
to it the operator 
$\alpha:\FL(\mathbb{R}^n)\to\FL(\mathbb{R}^n)$ 
given by
\begin{equation}    \label{eq4.2}
(\alpha F)(y)=\Big(\frac{1}{\sqrt{2\pi}}\Big)^n\!\int
\alpha(\omega)\,\fourier{F}(\omega)\,
\mathrm{e}^{\,\mathrm{i}\,\omega\cdot y}\domega.
\end{equation}
The function $\widetilde{U}=\alpha F$ can then
serve as a first approximation of the solution
(\ref{eq3.4}). In general, this approximation 
will not be very precise but can be 
iteratively improved. Starting from $U_0=0$
or $U_1=\alpha F$, let
\begin{equation}    \label{eq4.3}
U_{k+1}=(I-\alpha\L)U_k+\alpha F.
\end{equation}
The convergence of these iterates to the solution 
(\ref{eq3.4}) of the equation $\L U=F$ can already 
be shown under these very modest assumptions.

\begin{theorem}     \label{thm4.1}
The iterates \rmref{eq4.3} converge in the norm 
\rmref{eq2.2} to the solution \rmref{eq3.4}. 
\end{theorem}
\begin{proof}
The errors possess the representation
\begin{displaymath}    
U-U_k=(I-\alpha\L)^k\,U.
\end{displaymath}
The $L_1$-norm of the Fourier transform 
of the errors is therefore
\begin{displaymath}
\|\fourier{U}-\fourier{U}_k\|_{L_1}=\int
\Big|\Big(1-\alpha(\omega)\big(\mu+\|T^t\omega\|^2\big)\Big)^k
\,\fourier{U}(\omega)\Big|\domega.
\end{displaymath}
The integrands are by the assumption made above
bounded by the absolute value of the integrable 
function $\fourier{U}$ and tend almost everywhere 
to zero as $k$ goes to infinity. The dominated 
convergence theorem yields therefore
\begin{displaymath}
\lim_{k\to\infty}\|\fourier{U}-\fourier{U}_k\|_{L_1}=0,
\end{displaymath}
which proves the proposition.
\qed
\end{proof}
The same kind of result obviously also holds 
in the norm
\begin{equation}    \label{eq4.4}
\|U\|=\Big(\frac{1}{\sqrt{2\pi}}\Big)^n\!\int
\big(\mu+\|T^t\omega\|^2\big)|\fourier{U}(\omega)|\domega,
\end{equation}
by which the derivatives up to second order of 
the trace of $U$ can be estimated, and in all 
norms of similar type, be they based on the 
$L_1$- or the $L_2$-norm of the Fourier transform,
corresponding properties of the right-hand side 
provided, of course. An interesting example is 
the Hilbert space norm given by the expression
\begin{equation}    \label{eq4.5}
\|U\|^2=\,\int|\fourier{U}(\omega)|^2\,
\prod_{i=1}^n\frac{1+\omega_i^2}{2}\,\domega,
\end{equation}
which dominates the $L_1$-based norm used so 
far and thus also the maximum norm. It measures 
the size of the first order mixed derivatives 
and is tensor compatible.

\pagebreak

Provided that on the support of $\fourier{F}$ 
and thus also of the solution and the iterates
\begin{equation}    \label{eq4.6}
\Big|\,1-\alpha(\omega)\big(\mu+\|T^t\omega\|^2\big)\,\Big|
\,\leq\,q\,<\,1
\end{equation}
holds, one obtains from the Fourier representation 
of the errors the estimate
\begin{equation}    \label{eq4.7}
\|U-U_k\|\leq q^k\|U\|
\end{equation}
in the norms (\ref{eq2.2}) and (\ref{eq4.4}). 
In this case, one can use polynomial acceleration 
to speed up the convergence of the iteration or, 
in other words, to improve the quality of the 
approximations of the solution. That is, one 
replaces the $k$-th iterate by a weighted mean 
of the iterate itself and all previous ones. 
The errors of the recombined iterates possess 
then again a Fourier representation 
\begin{equation}    \label{eq4.8}    
(U-U_k)(y)=\Big(\frac{1}{\sqrt{2\pi}}\Big)^n\!\int
P_k\Big(\alpha(\omega)\big(\mu+\|T^t\omega\|^2\big)\Big)\,
\fourier{U}(\omega)\,
\mathrm{e}^{\,\mathrm{i}\,\omega\cdot y}\domega,
\end{equation}
but now not with the polynomials $P_k(\lambda)=(1-\lambda)^k$
but polynomials
\begin{equation}    \label{eq4.9}   
P_k(\lambda)=\sum_{\ell=0}^k\alpha_{k\ell}(1-\lambda)^\ell,
\quad
\sum_{\ell=0}^k\alpha_{k\ell}=1.
\end{equation}
Let $T_k$ denote the Chebyshev polynomial of degree $k$.
Among all polynomials $P$ of degree $k$ that satisfy the 
normalization condition $P(0)=1$, the polynomial
\begin{equation}    \label{eq4.10}
P_k(\lambda)=T_k\bigg(\dfrac{b+a-2\lambda}{b-a}\bigg)
\bigg/T_k\bigg(\dfrac{b+a}{b-a}\bigg)
\end{equation}
is then the only one that attains on a given interval 
$0<a\leq\lambda\leq b$ the smallest possible maximum 
absolute value, which is, in terms of the ratio 
$\kappa=b/a$, given by
\begin{equation}    \label{eq4.11}
\max_{a\leq\lambda\leq b}|P_k(\lambda)|=\frac{2r^k}{1+r^{2k}},
\quad r=\frac{\sqrt{\kappa}-1}{\sqrt{\kappa}+1}.
\end{equation}
As is well-known, this property plays a central role 
in the analysis of the conjugate gradient method. 
In our case we have $a=1-q$ and $b=1+q$. Inserting 
the corresponding polynomials $P_k$, one obtains 
the error estimate 
\begin{equation}    \label{eq4.12}    
\|U-U_k\|\leq
\frac{2r^k}{1+r^{2k}}\,\|U\|
\end{equation}
for the recombined iterates. This is in 
comparison to the convergence rate
\begin{equation}    \label{eq4.13}    
q=\frac{\kappa-1}{\kappa+1}
\end{equation}
of the original iteration a potentially 
big and very substantial improvement.


\section{A particular class of iterative methods}

Let $\rho:\mathbb{R}^n\to\mathbb{R}$ be a measurable,
locally bounded function for which
\begin{equation}    \label{eq5.1}
\|T^t\omega\|\leq\rho(\omega)
\end{equation}
holds for all $\omega\in\mathbb{R}^n$. We assign to this 
$\rho$ the iteration (\ref{eq4.3}) based on the function
\begin{equation}    \label{eq5.2}
\alpha(\omega)=\frac{1}{\mu+\rho(\omega)^2}.
\end{equation}
Our basic example is the function 
$\rho(\omega)=\|T\|\|\omega\|$. The function (\ref{eq5.2}) 
can then in a second step, as indicated in the introduction 
and explained below in more detail, again be approximated 
by a sum of Gauss functions. In any case,
\begin{equation}    \label{eq5.3}
1-\alpha(\omega)\big(\mu+\|T^t\omega\|^2\big)=
\frac{\rho(\omega)^2-\|T^t\omega\|^2}{\rho(\omega)^2+\mu}.
\end{equation}
That is, the condition (\ref{eq4.1}) is because of 
(\ref{eq5.1}) everywhere satisfied. The error thus 
tends by Theorem~\ref{thm4.1} in the norm 
(\ref{eq2.2}) to zero. The same holds for the norm 
(\ref{eq4.4}), and for any other norm like that 
given by (\ref{eq4.5}), corresponding regularity 
properties of the solution provided. The key feature 
for understanding the convergence of the iteration 
comes from the analysis of the sets
\begin{equation}    \label{eq5.4}
S(\delta)=
\big\{\omega\,\big|\,\|T^t\omega\|\geq\delta\rho(\omega)\big\},
\quad 0\leq\delta\leq 1.
\end{equation}
If $\rho(\omega)$ is a norm or seminorm, $S(\delta)$ 
is a cone, that is, with $\omega$ every scalar multiple 
of $\omega$ is contained in $S(\delta)$. Independent 
of the size of $\mu$, on the set $S(\delta)$
\begin{equation}    \label{eq5.5}
\delta^2\leq\alpha(\omega)\big(\mu+\|T^t\omega\|^2\big)\leq 1
\end{equation}
holds. If the Fourier transform of the right-hand side 
$F$ of the equation (\ref{eq3.6}) and with that also 
the Fourier transform of its solution $U$ vanish 
outside of the set $S(\delta)$, this implies the error 
estimate
\begin{equation}    \label{eq5.6}
\|U-U_k\|\leq q^k\|U\|, \quad q=1-\delta^2.
\end{equation}
The crucial point is that the regions $S(\delta)$ 
fill in case of high dimensions $m$ almost the 
complete space once $\delta$ falls below a certain 
bound, without any sophisticated further adaption 
of the function $\rho(\omega)=\|T\|\|\omega\|$
to the matrix $T$.

\begin{theorem}     \label{thm5.1}
Let $\rho(\omega)=\|T\|\|\omega\|$ and let $\kappa$ 
be the ratio of the maximum and the minimum singular
value of the matrix $T$, the condition number of 
$T$. If $\kappa\delta<1$, then
\begin{equation}    \label{eq5.7}
\frac{\lambda\big(
\big\{\omega\,\big|\,\|T^t\omega\|<\delta\rho(\omega),\,
\|\omega\|\leq R\big\}\big)}
{\lambda\big(\big\{\omega\,\big|\,\|\omega\|\leq R\big\}\big)}
\leq 
\psi\bigg(\frac{\kappa\delta}{\sqrt{1-\kappa^2\delta^2}}\bigg)
\end{equation}
holds for all $R>0$, where $\lambda$ is the volume measure 
on the $\mathbb{R}^n$ and
\begin{equation}    \label{eq5.8}
\psi(\varepsilon)=
\frac{2\,\Gamma(n/2)}{\Gamma(m/2)\Gamma((n-m)/2)}
\int_0^{\varepsilon}\frac{t^{m-1}}{(1+t^2)^{n/2}}\dt.
\end{equation}
Equality holds if and only if $\kappa=1$, that is, 
if all singular values of $T$ coincide.
\end{theorem}
\begin{proof}
Differing from the notation in the theorem but consistent 
within the proof, we split the vectors in $\mathbb{R}^n$ 
into parts $\omega\in\mathbb{R}^m$ and 
$\eta\in\mathbb{R}^{n-m}$. Starting point of our 
argumentation is a singular value decomposition 
$T^t=U\Sigma V^t$ of $T^t$. As the multiplication with 
the orthogonal matrices $U$ and $V^t$, respectively, 
does not change the euclidian length of a vector, the 
set whose volume has to be estimated consists of the 
points with components $\omega$ and $\eta$ in the ball 
of radius $R$ around the origin for which
\begin{displaymath}
\Big\|\Sigma V^t\oe\Big\|<\delta\,\|T\|\Big\|V^t\oe\Big\|
\end{displaymath}
holds. Because the volume is invariant under orthogonal 
transformations, the volume of this set coincides 
with the volume of the set of all points in the ball 
for which 
\begin{displaymath}
\Big\|\Sigma \oe\Big\|<\delta\,\|T\|\Big\|\oe\Big\|
\end{displaymath}
holds. Let $0<\sigma_1\leq\ldots\leq\sigma_m$ be 
the diagonal elements of the diagonal matrix
$\Sigma$, the singular values of the transpose 
$T^t$ of $T$ and of $T$ itself, and let 
$\kappa=\sigma_m/\sigma_1$. Since 
\begin{displaymath}
\sigma_1\|\omega\|\leq\Big\|\Sigma \oe\Big\|, 
\quad \|T\|=\sigma_m,
\end{displaymath}
the given set is a subset of the set of all those 
points in the ball for which
\begin{displaymath}
\|\omega\|<\kappa\delta\,\Big\|\oe\Big\|
\end{displaymath}
holds. If $\kappa=1$, that is, if all singular 
values of the matrix $T$ are equal, 
\begin{displaymath}
\sigma_1\|\omega\|=\Big\|\Sigma \oe\Big\|. 
\end{displaymath}
The two sets then coincide and nothing is lost up 
to here. If $\kappa>1$, there exists a vector with
components $\omega_0$ and $\eta_0$ inside of the 
ball under consideration for which
\begin{displaymath}
\Big\|\Sigma \Big(\begin{matrix}\omega_0\\\eta_0\end{matrix}\Big)\Big\|
=\sigma_m\|\omega_0\|, \quad
\frac{1-\kappa^2\delta^2}{\kappa^2\delta^2}\,\|\omega_0\|^2
<\|\eta_0\|^2< \frac{1-\delta^2}{\delta^2}\,\|\omega_0\|^2
\end{displaymath}
holds. For this vector and thus also for all 
vectors sufficiently close to it we have 
\begin{displaymath}
\Big\|\Sigma \oe\Big\|>\delta\,\|T\|\Big\|\oe\Big\|,
\quad
\|\omega\|<\kappa\delta\,\Big\|\oe\Big\|.
\end{displaymath}
That means that the two sets then differ and that
the second one has a greater volume. In what follows, 
we will calculate the volume of the latter one and 
compare it with the volume of the ball. We can
restrict ourselves here to the radius $R=1$. Let
\begin{displaymath}
\varepsilon=\frac{\kappa\delta}{\sqrt{1-\kappa^2\delta^2}}.
\end{displaymath}
The set consists then of the points in 
the $n$-dimensional unit ball for which
\begin{displaymath}
\|\omega\|<\varepsilon\|\eta\|
\end{displaymath}
holds. Its volume can, by Fubini's theorem, 
be expressed as double integral 
\begin{displaymath}
\int\bigg(\int H\big(\varepsilon\|\eta\|-\|\omega\|\big)
\chi\big(\|\omega\|^2+\|\eta\|^2\big)\domega
\bigg)\deta,
\end{displaymath}
where $H(t)=0$ for $t\leq0$, $H(t)=1$ for $t>0$,
$\chi(t)=1$ for $t\leq 1$, and $\chi(t)=0$ for 
arguments $t>1$. It tends, by the dominated
convergence theorem, to the volume of the
unit ball as $\varepsilon$ goes to infinity.
In terms of polar coordinates, it reads
\begin{displaymath}
(n-m)\nu_{n-m}\int_0^\infty\bigg(
m\nu_{m}\int_0^{\varepsilon s}\chi\big(r^2+s^2)r^{m-1}\dr
\bigg)s^{n-m-1}\ds,
\end{displaymath}
with $\nu_d$ the volume of the $d$-dimensional unit 
ball. Substituting $t=r/s$ in the inner integral 
and interchanging the order of integration, it 
attains finally the value
\begin{displaymath}
\frac{(n-m)\nu_{n-m}\,m\nu_m}{n}
\int_0^{\varepsilon}\frac{t^{m-1}}{(1+t^2)^{n/2}}\dt.
\end{displaymath}
Dividing this by the volume $\nu_n$ of the 
unit ball itself and remembering that 
\begin{displaymath}
\nu_d=\frac{2}{d}\,\frac{\pi^{d/2}}{\Gamma(d/2)},
\end{displaymath}
this completes the proof of the estimate (\ref{eq5.7}) 
and shows that equality holds if and only if $\kappa=1$, 
that is, if all singular values of $T$ coincide. 
Moreover, we have shown that the function (\ref{eq5.8}) 
tends to one as $\varepsilon$ goes to infinity, and 
the bound for the ratio of the two volumes itself to 
one as $\kappa\delta$ goes to one.
\qed
\end{proof}

The theorem states that ratio of the two volumes 
tends like $\delta^m$ to zero as $\delta$ goes
to zero. It can also for rather large $\delta$ 
still attain extremely small values, which means 
that in most of the frequency space $\|T^t\omega\|$ 
behaves like the norm of $\omega$. For $m=128$ 
and $n=256$, for example, the ratio is for 
$\kappa\delta\leq 1/4$ less than $1.90\cdot 10^{-42}$, 
and for $\kappa\delta\leq 1/2$ still less than 
$6.95\cdot 10^{-10}$. Figure~1 shows the bound 
(\ref{eq5.7}) as a function of $\kappa\delta<1$ 
for $n=2m$, with $m=2,4,8,\ldots,512$.
\begin{figure}[t]   \label{fig1}
\includegraphics[width=0.94\textwidth]{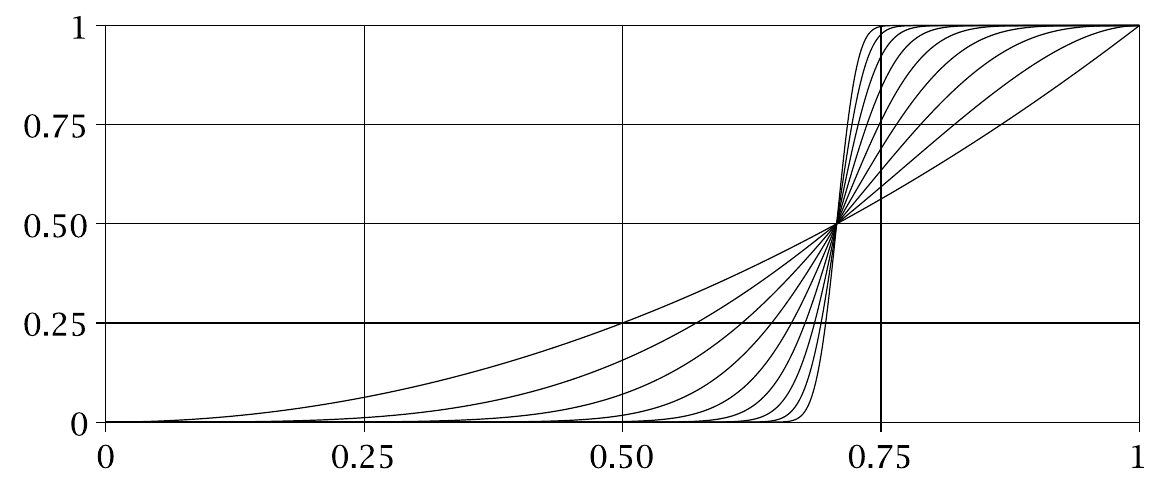}
\caption{The bound (\ref{eq5.7}) as function of
 $0<\kappa\delta<1$ for dimensions 
$m=2,4,8,\ldots,512$ and $n=2m$}
\end{figure}
As long as the Fourier transform of the solution 
is not strongly concentrated around the kernel 
of $T^t$, the regions of worse convergence will 
in such cases hardly affect the iterates and 
can be ignored.

\pagebreak

Polynomial acceleration as in Sect.~4 does not change 
this picture. If one inserts in (\ref{eq4.8}) the 
polynomials (\ref{eq4.10}) that attain the smallest 
possible maximum absolute value on the interval 
$a\leq\lambda\leq b$ with $a=\delta^2$ and $b=1$,
one gets the error estimate
\begin{equation}    \label{eq5.9}    
\|U-U_k\|\leq\frac{2q^k}{1+q^{2k}}\,\|U\|, 
\quad q=\frac{1-\delta}{1+\delta},
\end{equation}
for the recombined iterates as long as the Fourier 
transform of the solution $U$ vanishes outside of 
$S(\delta)$. In the general case, the corresponding 
part of the error is thus reduced by a much larger 
factor than with the basic iteration. As the 
polynomials (\ref{eq4.10}) take on the interval
$0<\lambda<a$ values $0<P_k(\lambda)<1$, 
and satisfy there the estimate
\begin{equation}    \label{eq5.10}    
0<P_k(\lambda)<\bigg(1-\frac{2\lambda}{a+b}\bigg)^k,
\end{equation}
the remaining part of the error does not blow up 
and even tends to zero as $k$ goes to infinity. 
The iteration error can thus also in the described 
cases be reduced very substantially replacing the 
original iterates by linear combinations of all 
iterates up to the given one.  
These arguments also apply when the function 
(\ref{eq5.2}), here
\begin{equation}    \label{eq5.11}    
\alpha(\omega)=\frac{1}{\mu+\|T\|^2\|\omega\|^2},
\end{equation}
is approximated by a sum $\widetilde{\alpha}(\omega)$ 
of Gauss functions that satisfies an estimate
\begin{equation}    \label{eq5.12}
0\leq\widetilde{\alpha}(\omega)\leq(1+\varepsilon)\alpha(\omega)
\end{equation}
on the whole frequency space and the reverse estimate
\begin{equation}    \label{eq5.13}
(1-\varepsilon)\alpha(\omega)\leq\widetilde{\alpha}(\omega)
\end{equation}
on a ball around the origin. The radius of this ball 
determines the spatial resolution and must therefore 
be chosen sufficiently large. Approximations that 
meet these requirements can be constructed truncating 
the infinite series (\ref{eq1.5}) mentioned in the 
introduction; see the appendix for details. The part 
of the error with Fourier transform supported on the 
intersection of the set $S(\delta)$ and this ball 
tends then, for the choice $a=(1-\varepsilon)\delta^2$ 
and $b=1+\varepsilon$ in the polynomials (\ref{eq4.10}), 
again rapidly to zero, from one iteration step to 
the next asymptotically at least by the factor 
\begin{equation}    \label{eq5.14}    
q'=\frac{1-\delta'}{1+\delta'}, \quad
\delta'=\sqrt{\frac{1-\varepsilon}{1+\varepsilon}}\,\delta,
\end{equation}
which differs also for comparatively crude approximations 
of the function (\ref{eq5.11}) not substantially from 
the factor in (\ref{eq5.9}) for the unperturbed case.

Assume that the approximation $\widetilde{\alpha}(\omega)$
of the function (\ref{eq5.11}) is based on an approximation
of $1/r$ with relative accuracy $\varepsilon$ for
$\mu\leq r\leq\mu R$, where $R$ is a large number, say 
$R=10^{12}$ or $10^{18}$. The estimate (\ref{eq5.13}) 
holds then on the ball
\begin{equation}    \label{eq5.15}
B=\big\{\omega\,\big|\,\mu+\|T\|^2\|\omega\|^2\leq\mu R\big\}.
\end{equation}
The Fourier transform of the solution (\ref{eq3.4}) 
of the equation (\ref{eq3.6}) can in this case on 
the complement $S(\delta)\setminus B$ of the set
$S(\delta)\cap B$ be pointwise estimated as
\begin{equation}    \label{eq5.16}
|\fourier{U}(\omega)|\leq
\frac{1}{\delta^2\mu R}\,|\fourier{F}(\omega)|
\end{equation}
by the Fourier transform of the right-hand side 
of the equation. This means that already due 
to the inherent smoothing properties of the
equation not much is lost when also the part
of the error with Fourier transform supported 
on $S(\delta)\setminus B$ is ignored and 
(\ref{eq5.13}) holds only on the given ball 
and not on the whole frequency space.

The topic of this paper is structural properties 
of the solutions of the differential equation
(\ref{eq1.1}). Our results open the possibility 
to apply tensor product methods to the approximation 
of solutions that themselves do not possess a tensor 
product structure and are not well separable. The 
practical feasibility of the approach depends on 
the representation of the involved tensors and 
of the factors of which they are composed, in 
particular on the access to their Fourier transform 
or the difficulty to calculate their convolution 
with a Gauss function. An interesting case is 
when these factors are themselves expanded into 
Gauss-Hermite functions. The iteration steps 
(\ref{eq4.3}) then do not lead out of this class 
of functions. This enables a very efficient
realization. What remains is the question how and 
to what extent the iterates can be compressed in 
between to keep the amount of work and storage 
under control without affecting the accuracy to much. 
We refer to \cite{Dahmen-DeVore-Grasedyck-Sueli} 
for such considerations.


\section{The limit behavior for fixed dimension ratios}
    
Figure 1 suggests that the bound (\ref{eq5.7}) 
tends to zero for all $\kappa\delta$ below some 
jump discontinuity and to one for the 
$\kappa\delta$ above this point if the ratio 
of the dimensions $m$ and $n$ is kept fixed 
and $m$ tends to infinity. This is indeed 
the case.

\begin{theorem}       \label{thm6.1}       
If one keeps the ratio $m/n$ of the dimensions 
fixed and lets $m$ tend to infinity, the functions 
\rmref{eq5.8} tend for arguments $\varepsilon$ 
left of the jump discontinuity
\begin{equation}    \label{eq6.1}    
\varepsilon_0=\sqrt{\frac{m}{n-m}}
\end{equation}
pointwise to zero and for arguments $\varepsilon$ 
right of it pointwise to one. The maximum distance 
of the function values $\psi(\varepsilon)$ to zero 
and one, respectively, tends for all $\varepsilon$ 
outside any given interval around the jump
discontinuity exponentially to zero. 
\end{theorem} 

\begin{proof}
Stirling's formula states that there is a function 
$0<\mu(x)<1/(12x)$ such that
\begin{displaymath}
\ln\Gamma(x)=\frac{2x-1}{2}\ln(x)-x+\ln(\sqrt{2\pi})+\mu(x)
\end{displaymath}
holds for all arguments $x>0$; see \cite[Eq.~5.6.1]{DLMF} 
and \cite{Koenigsberger} for a proof. Independent of the
size and the ratio of the dimensions $m$ and $n$, it leads 
to the representation 
\begin{displaymath}
\ln\bigg(\frac{2\,\Gamma(n/2)}{\Gamma(m/2)\Gamma((n-m)/2)}\bigg)
=\beta m+\frac{\ln(m)}{2}-\beta_0-\ln(\sqrt{\pi})+\beta_1
\end{displaymath}
of the logarithm of the prefactor, where the quantities
$\beta$ and $\beta_0$ are given by
\begin{gather*}
2\beta=\frac{n}{m}\,\ln\left(\frac{n}{m}\right)-
\left(\frac{n}{m}-1\right)\ln\left(\frac{n}{m}-1\right), 
\\
2\beta_0=\ln\left(\frac{n}{m}\right)-\ln\left(\frac{n}{m}-1\right)
\end{gather*}
and the terms coming from the function $\mu(x)$ are 
collected in the remainder 
\begin{displaymath}
\beta_1=\mu\Big(\frac{n}{2}\Big)-\mu\Big(\frac{m}{2}\Big)
-\mu\Big(\frac{n-m}{2}\Big).
\end{displaymath}
It tends in the present case of a fixed ratio $m/n$ like
\begin{displaymath}
\beta_1=\mathcal{O}\bigg(\frac{1}{m}\bigg)
\end{displaymath}
to zero as the dimension $m$ goes to infinity.

We keep the integers $m$ and $n$ in the following 
fixed, assume that they are relatively prime, 
and study with help of this representation the 
limit behavior of the functions
\begin{displaymath}   
\psi_k(\varepsilon)=
\frac{2\,\Gamma(kn/2)}{\Gamma(km/2)\Gamma((kn-km)/2)}
\int_0^{\varepsilon}
\frac{t^{km-1}}{(1+t^2)^{kn/2}}\dt
\end{displaymath}
as $k$ goes to infinity. The with the prefactor 
multiplied integrands can be written as
\begin{displaymath}
C(k)\sqrt{k}\;
\bigg(\frac{\mathrm{e}^{\,\beta m}t^m}{(1+t^2)^{n/2}}\bigg)^{k-1}\!
\frac{t^{m-1}}{(1+t^2)^{n/2}},
\end{displaymath}
where $C(k)$ remains bounded and tends to 
the limit 
\begin{displaymath}
\lim_{k\to\infty} C(k)=
\sqrt{\frac{m}{\pi}}\;\mathrm{e}^{\,\beta m-\beta_0}.
\end{displaymath}
The term in the brackets attains its global maximum 
at the point $t=\varepsilon_0$ specified above. As 
it increases strictly for $t<\varepsilon_0$, takes 
the value one at $t=\varepsilon_0$, and decreases 
strictly for $t>\varepsilon_0$, there exists for 
every open interval around $\varepsilon_0$ a 
$q<1$ with
\begin{displaymath}
0\leq\frac{\mathrm{e}^{\,\beta m}t^m}{(1+t^2)^{n/2}}\leq q
\end{displaymath}
for all $t\geq 0$ outside of it. As stated in the 
proof of Theorem~\ref{thm5.1}, $\psi_k(\varepsilon)$ 
tends to one as $\varepsilon$ goes to infinity. 
For arguments $\varepsilon$ right of the interval 
therefore
\begin{displaymath}
0<1-\psi_k(\varepsilon)<C(k)\sqrt{k}\,q^{k-1}
\int_{\varepsilon}^\infty\frac{t^{m-1}}{(1+t^2)^{n/2}}\dt
\end{displaymath}
holds. For arguments $\varepsilon>0$ left of 
the interval one obtains
\begin{displaymath}
0<\psi_k(\varepsilon)<C(k)\sqrt{k}\,q^{k-1}
\int_0^{\varepsilon}\frac{t^{m-1}}{(1+t^2)^{n/2}}\dt.
\end{displaymath}
As the integrals are uniformly bounded in
$\varepsilon$, this proves the proposition.
\qed
\end{proof}
If the ratio $m/n$ of the dimensions $m$ and $n$ 
is kept fixed and $m$ tends to infinity, the bound 
(\ref{eq5.7}) from Theorem~\ref{thm5.1} for the 
ratio of the two volumes tends therefore for values
of $\kappa\delta$ left of the jump discontinuity
\begin{equation}    \label{eq6.2}    
\xi_0=\sqrt{\frac{m}{n}} 
\end{equation}
to zero and for values of $\kappa\delta$ right of 
it to one, uniformly and exponentially outside 
every interval around $\xi_0$ and the faster,
the larger the interval is. For large dimensions 
$m$, the effective convergence rate in 
(\ref{eq5.9}) thus approaches the value
\begin{equation}    \label{eq6.3}    
q=\frac{1-\delta}{1+\delta}, \quad
\delta=\frac{1}{\kappa}\,\sqrt{\frac{m}{n}}. 
\end{equation}


\section{On the limit behavior in the general case }

In a similar way, one can estimate the bound 
from Theorem~\ref{thm5.1} by the leading term of 
its Taylor expansion at $\delta=0$. In contrast 
to the results from the previous section, this 
estimate does not rely on a given, fixed ratio 
of the dimensions $m$ and $n$.

\begin{theorem}     \label{thm7.1}
The bound from Theorem~{\rm \ref{thm5.1}} 
can, independent of the dimensions $m$ and 
$n\geq m+2$ and independent of their ratio, 
for $0<\kappa\delta<1$ be estimated as
\begin{equation}    \label{eq7.1}
\psi\bigg(\frac{\kappa\delta}{\sqrt{1-\kappa^2\delta^2}}\bigg)
\leq\, C\;\sqrt{\frac{n-m}{\pi mn}}\,
\bigg(\frac{\kappa\delta}{\delta_0}\bigg)^m
\end{equation}
by the leading term of its Taylor expansion at 
the point $\delta=0$, where $C$ remains bounded 
independent of $m$ and $n$ and tends to one
as $m$ and $n-m$ go to infinity. The scaling 
factor $\delta_0$ depends continuously on 
$m/n$ and possesses the representation
\begin{equation}    \label{eq7.2}
\delta_0=
\vartheta\Big(\frac{m}{n}\Big)\sqrt{\frac{m}{n}},
\end{equation}
with a function $\vartheta(x)$ that increases 
monotonously from $\vartheta(0)=1/\sqrt{\mathrm{e}}$ 
to $\vartheta(1)=1$.
\end{theorem}

\begin{proof}
We start from the observation that for 
dimensions $n\geq m+2$
\begin{displaymath}   
\psi\bigg(\frac{\kappa\delta}{\sqrt{1-\kappa^2\delta^2}}\bigg)
\leq\frac{2\,\Gamma(n/2)}{\Gamma(m/2)\Gamma((n-m)/2)}\,
\frac{(\kappa\delta)^m}{m}
\end{displaymath}
holds, which can be shown estimating the 
integrand in (\ref{eq5.8}) by the function
\begin{displaymath}   
\varphi(t)^{m-1}\varphi'(t), \quad
\varphi(t)=\frac{t}{\sqrt{1+t^2}}.
\end{displaymath}
That is, the bound from Theorem~\ref{thm5.1}  
can be estimated by the leading term of its 
Taylor expansion at the point $\delta=0$. 
Inserting the representation of the prefactor 
derived in the proof of Theorem~\ref{thm6.1}, 
one gets the estimate (\ref{eq7.1}), where 
$\ln(C)=\beta_1$ ranges between $-1/4$ and 
$1/18$ and tends to zero as $m$ and $n-m$ 
go to infinity. The scaling factor 
$\delta_0=\mathrm{e}^{-\beta}$ possesses 
the representation (\ref{eq7.2}), with 
\begin{displaymath}
\vartheta(x)=\exp\bigg(\frac{(1-x)\ln(1-x)}{2x}\bigg).
\end{displaymath}
The exponent can, for $|x|<1$, be expanded 
into the power series 
\begin{displaymath}
\frac{(1-x)\ln(1-x)}{2x}=-\frac12+
\frac12\sum_{k=1}^\infty\frac{1}{k(k+1)}\,x^k
\end{displaymath}
and tends to the limit value zero as $x$ 
goes to $1$.
\qed
\end{proof}
The estimate (\ref{eq7.1}) states that the bound 
from Theorem~\ref{thm5.1} rapidly decreases on 
the interval $\kappa\delta<\delta_0$ when the 
dimension $m$ increases. As long as the ratio 
$m/n$ remains fixed, the upper bound of this 
interval remains fixed and behaves qualitatively 
like the jump position (\ref{eq6.2}) from the 
previous section. If $n\sim m^2$, it decreases 
like $\sim 1/\sqrt{m}$, and if $n=m+m_0$ with 
a fixed $m_0$, it tends to one as $m$ goes to 
infinity.


\section{A final example}

Finally we return to the initially mentioned example 
of right-hand sides that depend explicitly on  
differences of the components $x_i$ of $x$. Let 
$\mathcal{I}_0$ be the set of the index pairs $(i,j)$, 
$i=1,\ldots,m-1$ and $j=i+1,\ldots,m$, and let
$\mathcal{I}$ be the subset of $\mathcal{I}_0$
assigned to the involved differences $x_i-x_j$. The 
total number of variables is then $n=m+|\mathcal{I}|$, 
with $|\mathcal{I}|$ the number of the index pairs 
in $\mathcal{I}$. We label the first $m$ components 
of the vectors in $\mathbb{R}^n$ by the indices 
$1,\ldots,m$, and the remaining components doubly,
by the index pairs in $\mathcal{I}$. The first $m$ 
components of $Tx$ are in this notation
\begin{equation}    \label{eq8.1}    
Tx|_i=x_i, \quad i=1,\ldots,m,
\end{equation}
and the remaining components are
\begin{equation}    \label{eq8.2}    
Tx|_{ij}=x_i-x_j, \quad (i,j)\in\mathcal{I}.
\end{equation}
The components of $T^t\omega$ can be calculated 
via the relation $T^t\omega\cdot x=\omega\cdot 
Tx$. If one sets $\omega_{ij}'=\omega_{ij}$ for 
the index pairs $(i,j)\in\mathcal{I}$ and 
otherwise formally $\omega_{ij}'=0$, they are
\begin{equation}    \label{eq8.3}    
T^t\omega|_i=\omega_i +
\sum_{j=1}^m(\omega_{ij}'-\omega_{ji}').
\end{equation}
Interesting is the case when $\omega$ is one 
the standard basis vectors $e_i$ or $e_{ij}$
for $i=1,\ldots, m$ and $(i,j)\in\mathcal{I}$,
respectively, pointing into the direction of 
the coordinate axes. Then
\begin{equation}    \label{eq8.4}
T^te_i|_i=e_i|_i,
\quad 
T^te_{ij}|_i=e_{ij}|_{ij}, \quad
T^te_{ij}|_j={}-e_{ij}|_{ij},
\end{equation}
and $T^t\omega|_k=0$ for all other components.
We have
\begin{equation}    \label{eq8.5}
\|T^te_i\|=\|e_i\|,\quad 
\|T^te_{ij}\|=\sqrt{2}\,\|e_{ij}\|.
\end{equation}
If $\delta\|T\|<1$, the coordinate axes hence 
are contained in the set $S(\delta)$ on which 
fast convergence is guaranteed, an advantageous 
property when the Fourier transforms of the 
functions under consideration are concentrated 
around them. This will, for example, be the 
case when their mixed derivatives are bounded, 
and in particular when the functions are 
tensor products of univariate functions.

Because of $\|Tx\|\geq\|x\|$ and $\|Te\|=\|e\|$ 
for $e=(1,\ldots,1)^t$, the minimum singular value 
of $T$ is one. Since the spectral norm and the 
spectral condition of $T$ therefore coincide, 
the estimate (\ref{eq5.7}) reduces here to
\begin{equation}    \label{eq8.6}
\frac{\lambda\big(
\big\{\omega\,\big|\,\|T^t\omega\|<\delta\|\omega\|,\,
\|\omega\|\leq R\big\}\big)}
{\lambda\big(\big\{\omega\,\big|\,\|\omega\|\leq R\big\}\big)}
\leq 
\psi\bigg(\frac{\delta}{\sqrt{1-\delta^2}}\bigg),
\end{equation}
where $\delta$ can attain values between zero 
and one. The norm of the matrix $T$ depends on the 
involved differences $x_i-x_j$, that is, on the 
set $\mathcal{I}$ of index pairs. Let $q_{ij}=1$ 
if either $(i,j)$ or $(j,i)$ belongs to 
$\mathcal{I}$, and $q_{ij}=0$ otherwise. Then
\begin{equation}    \label{eq8.7}  
\|Tx\|^2=\sum_{i=1}^m x_i\bigg(
\bigg(1+\sum_{j=1}^m q_{ij}\bigg)x_i-\sum_{j=1}^m q_{ij}x_j
\bigg).
\end{equation}
The diagonal and off-diagonal elements of $T^tT$
are therefore
\begin{equation}    \label{eq8.8}
T^tT|_{ii}=1+\sum_{j=1}^m q_{ij},\quad T^tT|_{ij}={}-q_{ij},
\end{equation}
and the row-sum norm of $T^tT$ assigned to the
maximum norm is
\begin{equation}    \label{eq8.9}
\|T^tT\|_\infty=\max_{i=1,\ldots,m}\bigg(1+2\sum_{j=1}^m q_{ij}\bigg).
\end{equation}
It represents an upper bound for the eigenvalues
of $T^tT$ and its square root therefore an upper
bound for the singular values of $T$. Like the 
ratio 
\begin{equation}    \label{eq8.10}
\frac{n}{m}=1+\frac{1}{2m}\sum_{i=1}^m d_i, \quad
d_i=\sum_{j=1}^m q_{ij},
\end{equation}
of the dimensions $n$ and $m$, the spectral 
norm and the spectral condition of $T$ can 
thus be bounded in terms of the degrees 
$d_i$ of the vertices of the underlying graph.


\section*{Appendix. On the exponential approximation}

The approximation (\ref{eq1.5}) of $1/r$ by exponential
functions can be written in the form
\begin{equation*}
\frac{1}{r}\approx
h\sum_{k=-\infty}^\infty\mathrm{e}^{kh}\exp(-\mathrm{e}^{kh}r)
=\frac{\phi(\ln r)}{r},
\end{equation*}
where $\phi$ denotes the continuous, $h$-periodic function
\begin{equation*}
\phi(s)=h\sum_{k=-\infty}^\infty\exp\big(-\mathrm{e}^{kh+s}+kh+s\big).
\end{equation*}
It has been shown in \cite[Sect. 5]{Scholz-Yserentant} by 
means of tools from Fourier analysis that this function
approximates the constant $1$ with an absolute error
\begin{equation*}
\sim 4\pi h^{-1/2}\mathrm{e}^{-\pi^2/h}
\end{equation*}
as $h$ tends to zero, that is, already for rather large
values of $h$  with very high accuracy. The same kind
of representation holds when the series is replaced by 
a finite sum
\begin{equation*}
\frac{1}{r}\approx
h\sum_{k=k_1}^{k_2}\mathrm{e}^{kh}\exp(-\mathrm{e}^{kh}r)
=\frac{\widetilde{\phi}(\ln r)}{r}.
\end{equation*}
To find out with which relative error $\varepsilon$ this 	
sum approximates the function $1/r$ on a given interval 
$1\leq r\leq R$, and the correspondingly rescaled sum 
\begin{equation*}
\frac{1}{r}\approx
h\sum_{k=k_1}^{k_2}\frac{\mathrm{e}^{kh}}{\mu}
\exp\Big(-\frac{\mathrm{e}^{kh}}{\mu}\,r\Big)
\end{equation*}
the function then on the interval $\mu\leq r\leq R\mu$, 
thus one has to study the function
\begin{equation*}
\widetilde{\phi}(s)=
h\sum_{k=k_1}^{k_2}\exp\big(-\mathrm{e}^{kh+s}+kh+s\big).
\end{equation*}
If the approximation $\widetilde{\alpha}(\omega)$ of 
the function (\ref{eq5.11}) is based on the described
approximation of $1/r$ on the interval 
$\mu\leq r\leq R\mu$  , (\ref{eq5.12}) means that 
$\widetilde{\phi}(s)\leq 1+\varepsilon$ must hold 
for all arguments $s\geq 0$. The condition 
(\ref{eq5.13}) is satisfied for the $\omega$ 
in the ball 
\begin{equation*}    
\mu+\|T\|^2\|\omega\|^2\leq\mu R
\end{equation*}
if for all $s$ in the interval $0\leq s\leq\ln(R)$ 
conversely $1-\varepsilon\leq\widetilde{\phi}(s)$ 
holds. It does not require much effort to fulfill 
these conditions for moderate accuracies 
$\varepsilon$ and large~$R$, as needed in the
present context. The relative error $\varepsilon$ 
is, for instance, less than $0.01$ on the interval 
$\mu\leq r\leq 10^{18}\mu$ for $h=1.4$, $k_1=-35$, 
and $k_2=1$, that is, in the one-percent range
on an interval that spans eighteen orders of
magnitude.


\bibliographystyle{spmpsci}
\bibliography{references}

\begin{thebibliography}{1}
\providecommand{\url}[1]{{#1}}
\providecommand{\urlprefix}{URL }
\expandafter\ifx\csname urlstyle\endcsname\relax
  \providecommand{\doi}[1]{DOI~\discretionary{}{}{}#1}\else
  \providecommand{\doi}{DOI~\discretionary{}{}{}\begingroup
  \urlstyle{rm}\Url}\fi

\bibitem{Braess-Hackbusch}
Braess, D., Hackbusch, W.: Approximation of $1/x$ by exponential sums in
  $[1,\infty)$.
\newblock IMA J. Numer. Anal. \textbf{25}, 685--697 (2005)

\bibitem{Braess-Hackbusch_2}
Braess, D., Hackbusch, W.: On the efficient computation of high-dimensional
  integrals and the approximation by exponential sums.
\newblock In: R.~DeVore, A.~Kunoth (eds.) Multiscale, Nonlinear and Adaptive
  Approximation. Springer, Heidelberg (2009)

\bibitem{Dahmen-DeVore-Grasedyck-Sueli}
Dahmen, W., DeVore, R., Grasedyck, L., S{\"u}li, E.: Tensor-sparsity of
  solutions to high-dimensional elliptic partial differential equations.
\newblock Found. Comp. Math. \textbf{16}, 813--874 (2016)

\bibitem{Grasedyck}
Grasedyck, L.: Existence and computation of low {K}ronecker-rank approximations
  for large linear systems of tensor product structure.
\newblock Computing \textbf{72}, 247--265 (2004)

\bibitem{Hackbusch}
Hackbusch, W.: Tensor Spaces and Numerical Tensor Calculus.
\newblock Springer, Cham (2019)

\bibitem{Khoromskij}
Khoromskij, B.: Tensor-structured preconditioners and approximate inverse of
  elliptic operators in $\mathbb{R}^d$.
\newblock Constr. Approx. \textbf{30}, 599--620 (2009)

\bibitem{Koenigsberger}
K{\"o}nigsberger, K.: Analysis 1.
\newblock Springer, Berlin (2004)

\bibitem{DLMF}
Olver, F.W., Lozier, D.W., Boisvert, R.F., Clark, C.W. (eds.): NIST Handbook of
  Mathematical Functions.
\newblock Cambridge University Press, Cambridge (2010)

\bibitem{Scholz-Yserentant}
Scholz, S., Yserentant, H.: On the approximation of electronic wavefunctions by
  anisotropic {G}auss and {G}auss-{H}ermite functions.
\newblock Numer. Math. \textbf{136}, 841--874 (2017)

\end{thebibliography}


\end{document}